\documentclass[10pt,twocolumn]{article}

\usepackage[letterpaper,margin=0.75in]{geometry}
\usepackage[T1]{fontenc}
\usepackage[utf8]{inputenc}

\usepackage{amsmath}
\usepackage{amssymb}
\usepackage{mathtools}
\usepackage{amsthm}

\usepackage{newtxtext}
\usepackage{newtxmath}

\usepackage{graphicx}
\usepackage{subcaption}
\usepackage[dvipsnames]{xcolor}
\usepackage{microtype}
\usepackage{authblk}

\mathtoolsset{showonlyrefs}

\newtheorem{theorem}{Theorem}

\newtheorem{lemma}{Lemma}

\theoremstyle{definition}
\newtheorem{definition}{Definition}
\newtheorem{assumption}{Assumption}

\theoremstyle{remark}
\newtheorem{remark}{Remark}

\newcommand{\R}{\mathbb{R}}
\newcommand{\RR}{\mathbb{R}}
\newcommand{\EE}{\mathbb{E}}
\newcommand{\cP}{\mathcal{P}}
\newcommand{\drift}{b}
\newcommand{\tcost}{g}

\usepackage[
  colorlinks=true,
  linkcolor=blue,
  urlcolor=blue,
  citecolor=blue
]{hyperref}

\title{Extended Mean Field Control Games with Moment Interactions:
General Framework and Linear-Quadratic Model}

\author[1]{Zhongyuan Cao}
\author[1]{Mathieu Lauri{\`e}re}
\author[1]{Andrew Shi}
\author[1]{Jiefei Yang}

\affil[1]{NYU-ECNU Institute of Mathematical Sciences at NYU Shanghai,
567 West Yangsi Road, Shanghai 200126, People's Republic of China\\
\texttt{\{zc3151, mathieu.lauriere, andrewshi, jy5595\}@nyu.edu}}

\date{}

\begin{document}
\maketitle

\begin{abstract}
	Mean field control games (MFCGs) provide a framework for studying non-cooperative games involving groups of cooperative players in the limit as both the number of players and the number of groups go to infinity. 
    This class of games includes mean field games (MFGs) and mean field control (MFC) problems as special cases, corresponding to purely non-cooperative and purely cooperative settings, respectively. 
    We incorporate interactions through the mean of actions, hence the terminology of extended MFCGs. 
    We first show that the equilibrium control can be described using a coupled system of partial differential equations consisting of a Hamilton-Jacobi-Bellman equation and a Fokker-Planck equation. 
    After introducing the general framework, we focus on a linear-quadratic (LQ) structure. 
    We show that the equilibrium can be reduced to a system of ODEs, generalizing those obtained for MFGs and MFC problems. 
    We then provide two numerical examples illustrating specific features of MFCGs.
\end{abstract}

\section{Introduction}
\label{sec:introduction}

Mean Field Games (MFGs) and Mean Field Control (MFC) have been highly active areas of research since the foundational work of Lasry and Lions \cite{lasry2007MFG} and Huang, Caines, and Malhamé \cite{huang2007MFG}.
These frameworks capture the asymptotic behavior of large populations of interacting agents in purely non-cooperative (MFG) and purely cooperative (MFC) regimes.
However, many real-world systems do not fall strictly into either extreme.
Instead, they feature hierarchical structures where players coordinate their strategies within distinct groups, while the groups themselves compete non-cooperatively. 
When the number of groups is finite and each group is infinite, such games have been referred to as \emph{mean field type games}; see, e.g.,~\cite{djehiche2016mean,bacsar2026mean}. 
When the number of groups tends to infinity, this gives rise to MFGs in which each player solves an MFC problem.
This motivates the study of \textit{Mean Field Control Games} (MFCGs), introduced in~\cite{angiuli2023reinforcement}, which describe Nash equilibria among collaborating groups.
Recent literature has begun exploring models \cite{angiuli2022bank} and algorithms \cite{angiuli2026analysis} for this intermediate regime. 
Along this line, the authors of \cite{carmona2023nash} considered games that interpolate between MFC and MFG, and the authors of \cite{dayanikli2025cooperation} proposed models with a mixture of cooperation and non-cooperation.

Concurrently, a second major evolution in mean field theory has been the development of ``extended'' models.
Standard mean field frameworks assume that agents interact exclusively through the empirical distribution of their states. 
In many applications, however, agents' costs and dynamics depend explicitly on the \textit{actions} or \textit{controls} of the broader population.
To address this, extended MFGs were formalized in \cite{gomes2014extended, acciaio2019extended} to allow the Lagrangian to depend on the joint distribution of states and controls. 
This class has also been referred to as MFGs of controls; see~\cite{cardaliaguet2018mean,kobeissi2022classical} for analytical results and~\cite{achdou2021mean} for numerical approximations.
This has sparked significant recent work on extended mean field problems from both theoretical \cite{djete2022extended, djete2022mckean, bo2024extended, bensoussan2025extended} and numerical perspectives \cite{picarelli2025extended, reisinger2025convergence}. 
Related LQ models with state--control mean interactions have been
studied in~\cite{li2022dynamic,li2023linear,bensoussan2025linear}.
Most closely related, \cite{li2026generalized} studies a ``mean-field type game of controls'' whose representative player internalizes its own state and control expectations while treating the population averages as fixed, yielding the same nested MFC-best-response and mean-field-consistency structure as an LQ MFCG.

Our contribution is to formulate this structure first at the level
of a general nonlinear, moment-based HJB--FP system and then specialize
it to a multidimensional LQ model with distinct local and population
interaction coefficients. This differs from, e.g.,
\cite{graber2016linear,bensoussan2025linear}, which treats extended MFG and MFC
separately, and complements the common-noise FBSDE and master-equation
analysis of~\cite{li2026generalized}.
In particular, compared with \cite{li2022dynamic}, which considers a particular linear-quadratic control-average model, our formulation is developed first at the general PDE level and subsequently specialized to the LQ setting. 
The MFCG framework contains an interpolation of non-cooperative and cooperative regimes through separate group-level and population-level interaction terms, whereas the interpolation in \cite{carmona2023nash} is done at the level of the cost functional.

Specifically, we first formulate the general extended MFCG problem and characterize its equilibrium via a coupled system of partial differential equations, consisting of a Hamilton-Jacobi-Bellman (HJB) equation and a Fokker-Planck (FP) equation (see Theorem~\ref{thm:general_mfcg}).
Second, we specialize the framework to a Linear-Quadratic (LQ) structure.
We demonstrate that in the LQ setting, the PDE system can be reduced to a system of ordinary differential equations (ODEs) featuring Riccati-type equations (see Theorem~\ref{thm:lq_mfcg}).
Finally, we propose a model for LQ MFCGs with interpolation parameters that blend global and group-level benchmarks, and we illustrate the model through numerical experiments.

\section{Extended MFCG Model}
\label{sec:extended-mfcg-model}
 
In this section, we formulate the extended MFCG problem mathematically.

\noindent
\textbf{General notation.} Let $d$ be the state dimension, $k$ be the action dimension, and $p$ be the dimension of the interaction terms. 
$\cP(E)$ denotes the set of probability measures on $E$.

\noindent
\textbf{Individual dynamics.} The dynamics of a representative agent interact with the population through a time-dependent interaction term $\bar{\varphi}: [0,T] \to \RR^p$. 
If the agent uses a control $v:[0,T]\times \R^d \to \RR^k$, then we denote their state at time $t$ by $X^{v, \bar\varphi}_t \in \RR^d$. 
We denote by $m^{v, \bar\varphi}_t \in \cP(\R^d)$ the density of $X^{v, \bar\varphi}_t$.
We assume that the evolution of this state involves its own distribution through a term of the form:
\begin{align}
	\tilde{\varphi}_t^{v, \bar{\varphi}}
	 & = \EE[ \varphi(t,X^{v, \bar{\varphi}}_t, v(t,X^{v, \bar{\varphi}}_t))]
	\\
	 & = \int \varphi(t,x', v(t,x')) m^{v, \bar{\varphi}}(t,x') dx',
	\label{eq:tilde-phi-def}
\end{align}
where $\varphi: [0,T] \times \R^d \times \R^k \to \R^p$ is a given interaction function. The quantities $\bar{\varphi}$ and $\tilde{\varphi}$ play different roles. 
The interaction term $\tilde{\varphi}^{v, \bar{\varphi}}$ is computed from the state distribution generated by the representative group when it uses the control $v$, whereas $\bar{\varphi}$ is computed from the population distribution when every group uses the population strategy $\bar{v}$. 
At equilibrium, where $v = \bar{v}$, the two quantities coincide.
The dynamics are:
\begin{equation} \label{eq:state-dynamics}
	dX^{v, \bar{\varphi}}_t=\drift(t, X^{v, \bar{\varphi}}_t, v(t, X^{v, \bar{\varphi}}_t), \tilde{\varphi}_t^{v, \bar{\varphi}}, \bar{\varphi}_t)dt + \sigma(t, X^{v, \bar{\varphi}}_t)dW_t,
\end{equation}
with initial condition $X^{v, \bar{\varphi}}_0 \sim m_0$ for a given $m_0$. 
Note that these dynamics are of McKean-Vlasov (MKV) type. 

The density $m^{v, \bar{\varphi}}_t$ of $X^{v, \bar{\varphi}}_t$ satisfies the following Kolmogorov-Fokker-Planck (KFP) equation:
\begin{multline} \label{eq:KFP-eqn}
	\partial_t m^{v,\bar{\varphi}} + A^*m^{v,\bar{\varphi}} + \operatorname{div}\big(\drift(t, x, v, \tilde{\varphi}_t^{v,\bar{\varphi}}, \bar{\varphi}_t)m^{v,\bar{\varphi}}\big) = 0,
\end{multline}
with initial condition $m^{v,\bar{\varphi}}(0) = m_0$, where the second-order differential operator and its adjoint are defined by:
$
	Au(t,x) = -\operatorname{Tr}(a(t,x) D_x^2u(t,x)),
	A^*u(t,x) = -\sum_{i,j=1}^d \partial_{ij}\left( a_{ij}(t,x) u(t,x)\right),
$
with $a(t,x) = \frac{1}{2}\sigma\sigma^\top(t,x)$.

\noindent
\textbf{Population dynamics.} We assume that the interactions with the population are of the same form as~\eqref{eq:tilde-phi-def}. 
They could occur through a function different from $\varphi$, but, for simplicity, we consider the same function for individual and population interactions.
Hence, if the whole population uses control $\bar{v}$, the population density, denoted by $m^{\bar{v}}$, solves the KFP:
\begin{multline} \label{eq:KFP-eqn-pop}
	\partial_t m^{\bar{v}} + A^*m^{\bar{v}} + \operatorname{div}\big(\drift(t, x, \bar{v}, \tilde{\varphi}_t^{\bar{v}, \bar{\varphi}^{\bar{v}}}, \bar{\varphi}_t^{\bar{v}})m^{\bar{v}}\big) = 0,
\end{multline}
with initial condition $m^{\bar{v}}(0) = m_0$, where
\begin{align}
	\bar{\varphi}_t^{\bar{v}} & = \int  \varphi(t,x', \bar{v}(t,x')) m^{\bar{v}}(t,x') dx'.
\end{align}
When the population uses control $\bar{v}$, the term $\bar{\varphi}_t$ in~\eqref{eq:state-dynamics} is replaced by $\bar{\varphi}^{\bar{v}}_t$. Furthermore, to alleviate the notation, we denote $X^{v,\bar{v}}_t = X^{v,\bar{\varphi}^{\bar{v}}}_t$, $m^{v,\bar{v}}_t = m^{v,\bar{\varphi}^{\bar{v}}}_t$, $\tilde{\varphi}^{v,\bar{v}}_t = \tilde{\varphi}^{v,\bar{\varphi}^{\bar{v}}}_t$.

\noindent
\textbf{Cost function.}
If the population uses control $\bar{v}$ and the agent uses control $v$, the agent's cost is defined as follows; the integrals with respect to $m^{v,\bar v}$ can be interpreted as expectations with respect to $X^{v,\bar v}$:
\begin{align}
	J(v; \bar{v})
	 & = \int_0^T \int_{\R^d} f(t, x, v(t,x), \tilde{\gamma}_t^{v,\bar{v}}, \bar{\gamma}_t^{\bar{v}}) m^{v,\bar{v}}(t,x) dxdt
	\notag
	\\
	 & \qquad + \int_{\R^d} \tcost(x, \tilde{\psi}_T^{v,\bar{v}}, \bar{\psi}_T^{\bar{v}})m^{v,\bar{v}}(T,x)dx,
	\label{eq:objective-functional-1}
\end{align}
where the interaction terms are:
\begin{equation*}
	\begin{aligned}
		\tilde{\gamma}_t^{v, \bar{v}} & = \int_{\R^d}\gamma(t,x', v(t,x')) m^{v, \bar{v}}(t,x')dx', \\
		\bar{\gamma}_t^{\bar{v}}      & = \int_{\R^d} \gamma(t, x', \bar{v}(t,x'))m^{\bar{v}}(t,x')dx',     \\
		\tilde{\psi}_T^{v, \bar{v}}   & = \int_{\R^d}\psi(x') m^{v,\bar{v}}(T,x')dx',               \\
		\bar{\psi}_T^{\bar{v}}        & = \int_{\R^d} \psi(x')m^{\bar{v}}(T,x')dx'
	\end{aligned}
\end{equation*}
for some given functions $\gamma: [0,T] \times \R^d \times \R^k \to \R^p$ and $\psi: \R^d \to \R^p$.
The following solution concept is based on the idea of Nash equilibrium.
\begin{definition}[Extended MFCG Equilibrium] \label{def:equilibrium}
	A control $\hat{v}$ is an equilibrium for the extended MFCG problem if $\hat{v}$ is a minimizer of
	\(
	v \mapsto J(v; \hat{v}).
	\)
\end{definition}
In other words, a control is an equilibrium if it is an optimal control for a player facing a population in which all the players use this control. 
This can be interpreted as a fixed-point problem: first, given the population's evolution, we find an optimal control by solving an MFC problem; second, the population's evolution should be consistent with this optimal control.

\section{MFCG Equilibrium}
\label{sec:equilibrium}

In this section, we characterize the MFCG equilibrium by a forward-backward PDE system.

When the population's evolution is given, the problem for a representative player in MFCG reduces to an MFC problem. 
However, this problem is of ``extended'' type in the sense that it involves the action distribution. 
To the best of our knowledge, extended MFC problems have not yet been studied from a PDE perspective. 
For this reason, we start by establishing Lemma~\ref{lem:PDE-extended-MFC} in appendix, which provides a PDE system characterizing the solution of the individual MFC problem. 
Based on this, we prove the following theorem, which provides necessary conditions for an extended MFCG equilibrium.
\begin{theorem}
	\label{thm:general_mfcg}
	If $\hat{v}$ is an equilibrium control for the MFCG, then, for all $(t,x)$, it satisfies
	\begin{align}
		 & \partial_v f(t, x, \hat{v}(t,x), \bar{\gamma}_t^{\hat{v}}, \bar{\gamma}_t^{\hat{v}})
		\label{eq:optimality-condition-MFCG}
		\\
		 & + D_xu^{\hat{v}}(t,x) \cdot \partial_v \drift(t,x,\hat{v}(t,x), \bar{\varphi}_t^{\hat{v}}, \bar{\varphi}_t^{\hat{v}})
		\notag
		\\
		 & + \partial_v \gamma(t,x, \hat{v}(t,x)) \hat{K}^{\gamma}_t + \partial_v \varphi(t,x, \hat{v}(t,x)) \hat{K}^{\varphi}_t
		= 0,
		\notag
	\end{align}
	where $(m^{\hat{v}},u^{\hat{v}})$ solves the following PDE system:
	\begin{equation} \label{eq:MFCG-PDE}
		\begin{cases}
			\partial_t m^{\hat{v}}(t,x) + A^* m^{\hat{v}}(t,x)
			\\
			\qquad + \operatorname{div}(\drift(t,x,\hat{v}(t,x),\bar{\varphi}_t^{\hat{v}}, \bar{\varphi}_t^{\hat{v}}) m^{\hat{v}}(t,x)) = 0, \\
			m^{\hat{v}}(0, x) = m_0(x),                                                                                                 \\
			-\partial_t u^{\hat{v}}(t,x) + Au^{\hat{v}}(t,x)
			\\
			\qquad - \drift(t,x,\hat{v}(t,x),\bar{\varphi}_t^{\hat{v}}, \bar{\varphi}_t^{\hat{v}}) \cdot D_xu^{\hat{v}}(t,x)
			\\
			\qquad -f(t,x,\hat{v}(t,x), \bar{\gamma}_t^{\hat{v}}, \bar{\gamma}_t^{\hat{v}})
			\\
			\qquad
			- \varphi(t,x, \hat{v}(t,x)) \cdot \hat{K}^{\varphi}_t
			- \gamma(t,x, \hat{v}(t,x)) \cdot \hat{K}^{\gamma}_t
			=0 ,                                                                                                                        \\
			u^{\hat{v}}(T, x) = \tcost(x, \bar{\psi}^{\hat{v}}_T, \bar{\psi}^{\hat{v}}_T) + \psi(x) \cdot \hat{K}^{\psi}_T,
		\end{cases}
	\end{equation}
	with $\hat{K}^{\gamma}_t$, $\hat{K}^{\varphi}_t$, and $\hat{K}^{\psi}_T$ given by:
	\begin{small}
		\begin{align*}
			\hat{K}^{\gamma}_t  & = \int_{\R^d} \partial_{\tilde{\gamma}} f(t, x', \hat{v}(t,x'), \bar{\gamma}_t^{\hat{v}}, \bar{\gamma}_t^{\hat{v}}) m^{\hat{v}}(t,x') dx',                                    \\
			\hat{K}^{\varphi}_t & = \int_{\R^d} D_xu^{\hat{v}}(t, x') \cdot \partial_{\tilde{\varphi}} \drift(t,x',\hat{v}(t,x'), \bar{\varphi}_t^{\hat{v}}, \bar{\varphi}_t^{\hat{v}}) m^{\hat{v}}(t, x') dx', \\
			\hat{K}^{\psi}_T    & = \int_{\R^d} \partial_{\tilde{\psi}} \tcost(x', \bar{\psi}_T^{\hat{v}}, \bar{\psi}_T^{\hat{v}}) m^{\hat{v}}(T, x') dx'.
		\end{align*}
	\end{small}
\end{theorem}

\section{Linear-Quadratic Case}
\label{sec:LQ}

In this section, we consider the extended MFCG model in a linear-quadratic framework.
To alleviate the notation, we assume that the dynamics do not depend on mean field terms. 
Specifically, we take $\drift$ to be linear in $(x,v)$ and independent of $(\tilde{\varphi},\bar{\varphi})$, $f$ to be quadratic in $(x,v,\tilde{\gamma},\bar{\gamma})$, $\varphi$ and $\gamma$ to be linear in $(x,v)$, $\tcost$ to be quadratic in $(x,\tilde{\psi}, \bar{\psi})$, and $\psi$ to be linear in $x$.
To be specific, we consider the following model.
The dynamics are given by
\begin{equation} \label{eq:lq-dynamics}
	dX_t=(A_tX_t + B_tv_t)dt + C_tX_tdW_t,
\end{equation}
with $X_0 \sim m_0$, where $v_t=v(t,X_t)$ is a feedback control at time $t$ taking values in the action set $\mathbb{R}^k$, and $A$, $B$, and $C$ are given deterministic processes.
Let $\tilde{X}$ and $\tilde{v}$ denote the mean state and mean control within a group, and let $\bar{X}$ and $\bar{v}$ denote the corresponding population-wide means.
$\tilde{X}$ and $\bar{X}$ with respective dynamics~\eqref{eq:state-dynamics} and~\eqref{eq:lq-dynamics} remain distinct, because they aggregate at different levels and under different controls: $\tilde{X}_t$ is the mean state within the representative group, which uses the control $v$ and is therefore free to deviate, whereas $\bar{X}_t$ is the mean state across the whole population, every group of which uses $\bar{v}$. As with $\tilde{\varphi}$ and $\bar{\varphi}$, the two coincide at equilibrium, where $v = \bar{v}$.
Given the global population mean $\bar{X} = (\bar{X}_t)_{t \in [0,T]}$ and $\bar{v} = (\bar{v}_t)_{t\in [0,T]}$, a representative group planner chooses $v$ to minimize the total quadratic cost functional
\begin{multline} \label{eq:objective-functional}
	J(v; \bar{X}, \bar{v}) = \mathbb{E}\bigg[\int_0^T F(t, X_t, \tilde{X}_t, \bar{X}_t, v_t, \tilde{v}_t, \bar{v}_t)dt \\
		+ G(X_T, \tilde{X}_T, \bar{X}_T)\bigg],
\end{multline}
where the running cost $F$ and terminal cost $G$ are given by
\begin{align}
	 & F(t,x,\tilde{x}, \bar{x}, v, \tilde{v}, \bar{v})
	=  Q_tx\cdot x  +  \tilde{Q}_t\tilde{x}\cdot \tilde{x}
	+  \bar Q_t\bar{x}\cdot \bar{x}
	\\
	 & \qquad + M_t \bar{x}\cdot x  +  \tilde{M}_t \bar{x}\cdot  \tilde{x} + R_tv\cdot  v
	+\tilde{R}_t\tilde{v}\cdot  \tilde{v}
	\\
	 & \qquad +\bar{R}_t \bar{v}\cdot \bar{v}
	+ N_t \bar{v}\cdot v +  \tilde{N}_t \bar{v}\cdot \tilde{v}+ 2 S_tx\cdot v
	\\
	 & \qquad + 2 \tilde{S}_t \tilde{x}\cdot \tilde{v}  + 2\bar{S}_t \bar{x}\cdot \bar{v} + 2 q_t\cdot x + 2\tilde{q}_t\cdot \tilde{x}
	\\
	 & \qquad + 2\bar{q}_t\cdot\bar{x} + 2 r_t\cdot v + 2 \tilde{r}_t\cdot \tilde{v} + 2\bar{r}_t\cdot \bar{v},
	\\
	 & G(x, \tilde{x}, \bar{x}) =  Hx\cdot x + \tilde{H} \tilde{x}\cdot \tilde{x}+ \bar{H} \bar{x}\cdot\bar{x},
	\label{eq:LQ-costs}
\end{align}
where $Q$, $\tilde{Q}$, $\bar{Q}$, $M$, $\tilde{M}$, $R$, $\tilde{R}$, $\bar{R}$, $N$, $\tilde{N}$, $S$, $\tilde{S}$, $\bar{S}$, $q$, $\tilde{q}$, $\bar{q}$, $r$, $\tilde{r}$, $\bar{r}$ are deterministic, matrix-valued processes, and $H$, $\tilde{H}$, $\bar{H}$ are constant matrices. 
We denote by $\mathcal{S}^n$ the set of all $n\times n$ symmetric matrices with real entries.
For $U\in \mathcal{S}^n$, $U\geq 0$ means $U$ is positive semi-definite. 
The linear-quadratic framework studied in this paper differs from the control-average mean-field stochastic large-population system introduced in \cite{li2022dynamic}. A foundational two-Riccati decomposition for LQ mean-field control is developed in~\cite{yong2013linear}, while related LQ MFGs of controls are studied in~\cite{li2023linear}.
We derive the associated ODE system inspired by the results of \cite{graber2016linear}. 
Note that the author of \cite{graber2016linear} studies MFC and MFG problems separately, whereas we consider the MFCG formulation, which encompasses both MFC and MFG as special cases.

\begin{assumption}\label{assum:coeff}
	The coefficient matrices satisfy
	\begin{enumerate}
		\item $A, C\in L^{\infty}([0, T]; \mathbb{R}^{d \times d})$,
		      $B\in L^{\infty}([0, T]; \mathbb{R}^{d \times k})$;
		\item $Q, \tilde{Q},\bar{Q}, M, \tilde{M} \in L^{\infty}([0, T]; \mathcal{S}^d)$, $R,\tilde{R}, \bar{R}, N,\tilde{N} \in L^{\infty}([0, T]; \mathcal{S}^k)$, $H,\tilde{H}, \bar{H} \in \mathcal{S}^d$;
		\item $H \ge 0$, $H + \tilde{H} \ge 0$, and for some $\delta_1 \ge 0, \delta_2 > 0$, $Q, Q + \tilde{Q} \ge \delta_1 I$ and $R, R + \tilde{R} \ge \delta_2 I$;
		\item $S, \tilde{S},\bar{S} \in L^{\infty}([0, T]; \mathbb{R}^{k \times d})$; $q, \tilde{q},\bar{q} \in L^{\infty}([0, T]; \mathbb{R}^d)$; $r,\tilde{r}, \bar{r} \in L^{\infty}([0, T]; \mathbb{R}^k)$;
		\item $\|S\|_{\infty}^2, \|S + \tilde{S}\|_{\infty}^2 < \delta_1 \delta_2$ if $\delta_1 > 0$, and $S = \tilde{S} = 0$ otherwise.
	\end{enumerate}
\end{assumption}

In Assumption~\ref{assum:coeff}, boundedness of the coefficients guarantees well-posedness of the controlled state equation, while the positive definiteness conditions on the cost coefficients ensure strict convexity of the optimization problem and uniqueness of the optimal control.

We then derive the system of ordinary differential equations (ODEs) associated with the linear-quadratic MFCG equilibrium.
We first introduce two ODEs:

\begin{equation}
	\begin{cases}
		\dot{P}_t + A_t^\top P_t + P_t A_t + C_t^\top P_t C_t  + Q_t       \\
		\quad \quad  - \Lambda_0^\top(t)\Sigma_0^{-1}(t) \Lambda_0(t) = 0, \\
		P_T = H,
	\end{cases}
	\label{eq:P}
\end{equation}
and
\begin{equation}
	\begin{cases}
		\dot{\Pi}_t + A_t^\top\Pi_t + \Pi_tA_t + C_t^\top P_t C_t
		+ Q_t + \tilde{Q}_t                                   \\
		\quad \quad
		- \Lambda_1^\top(t)\Sigma_1^{-1}(t) \Lambda_1(t) = 0, \\
		\Pi_T = H + \tilde{H},
	\end{cases}
	\label{eq:Pi}
\end{equation}
where
$\Lambda_0(t)=B^\top_t P_t+ S_t,             
	 \Lambda_1(t)=B^\top_t \Pi_t+S_t+\tilde{S}_t, 
	 \Sigma_0(t)=R_t,                             
	 \Sigma_1(t)=R_t+ \tilde{R}_t.
$

Now we are ready to give a sufficient condition for the MFCG equilibrium.

\begin{theorem}
	\label{thm:lq_mfcg}
	Under Assumption~\ref{assum:coeff}, there exist unique solutions $P, \Pi$ to equations \eqref{eq:P} and \eqref{eq:Pi}.
	Moreover, if there exists a solution 
    to the following ODE system
	\begin{align}
		\begin{cases}
			\dot {z}_t=\big(A_t-B_t\Sigma_1^{-1}(t)\Lambda_1(t)\big)z_t- B_t \Sigma_1^{-1}(t)            \\
			\quad \quad \qquad \big(\frac 1 2(N_t+\tilde{N}_t)a_t+r_t+\tilde{r}_t + B_t^\top\phi_t\big), \\
			z_0 = \int_{\RR^d}x m_0(x)dx,                                                                \\
			a_t =-\Sigma_1^{-1}(t)\big(\Lambda_1(t)z_t+\frac 1 2(N_t+\tilde{N}_t)a_t                     \\
			\quad \quad +r_t+\tilde{r}_t+B^\top_t\phi_t\big),\\
             \dot\phi_t=
-\bigl(A_t^\top-\Lambda_1^\top\Sigma_1^{-1}B_t^\top\bigr)\phi_t-\frac{1}{2}(M^\top_t+\tilde{M}^\top_t)z_t\\
\quad \quad +\Lambda_1^\top(t)\Sigma_1^{-1}(t)\big(r_t+\tilde r_t+\frac{1}{2}(N^\top_t+\tilde{N}^\top_t)a_t\big)-(q_t+\tilde q_t),\\
\phi_T=0.
\end{cases}		\label{eq:coupled_z_a}
	\end{align}    
Then $(\hat{v}_t)_{t\in[0,T]}$ is a MFCG equilibrium control, where $\hat{v}$ is given by	
		$\hat{v}_t=-\Sigma_0^{-1}(t)\Lambda_0(t)(X_t-z_t)+a_t.$
\end{theorem}

\begin{figure*}[t]
	\centering
	\begin{subfigure}[t]{0.24\textwidth}
		\centering
		\includegraphics[width=\textwidth]{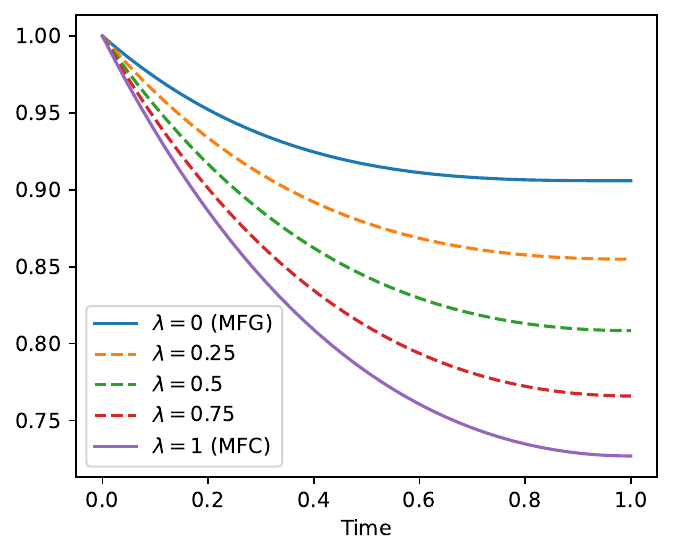}
	\end{subfigure}
	\hfill
	\begin{subfigure}[t]{0.24\textwidth}
		\centering
		\includegraphics[width=\textwidth]{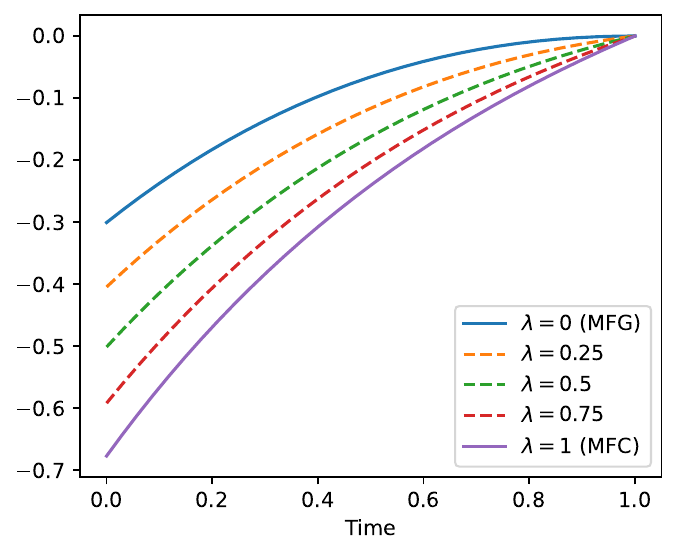}
	\end{subfigure}
	\hfill
	\begin{subfigure}[t]{0.24\textwidth}
		\centering
		\includegraphics[width=\textwidth]{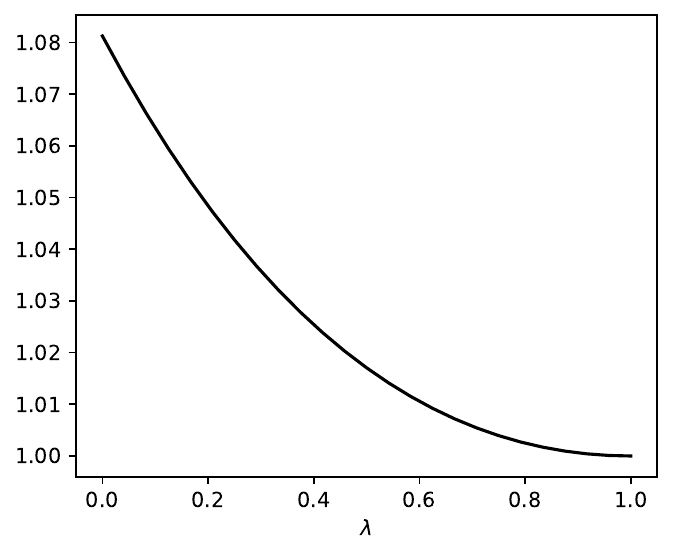}
	\end{subfigure}
	\caption{Experiment 1: mean state, mean control, and Price of Anarchy (PoA) under the modified benchmark cost~\eqref{eq:objective-ex1}. Left: Mean state $z_t$. Middle: Mean control $a_t$. Right: PoA $J_{\mathrm{MFCG}}/J_{\mathrm{MFC}}$.}
	\label{fig:experiment1}
\end{figure*}

\section{Model and Numerical Experiments}

We construct an MFCG parameterized by $(\lambda_1,\lambda_2)\in[0,1]^2$, allowing us to recover MFG and MFC as special cases.
We therefore consider a cost functional built around the following term:
  \begin{equation}\label{eq:lambda-interpolation}
  \resizebox{0.95\linewidth}{!}{$
  \begin{aligned}
  \frac{1}{2}\Bigl(X_t-(1-\lambda_1)\bar X_t-\lambda_1\widetilde X_t\Bigr)^2
  &+\frac{1}{2}\Bigl((1-\lambda_1)\bar X_t+\lambda_1\widetilde X_t\Bigr)^2\\
  +\frac{1}{2}\Bigl(\alpha_t-(1-\lambda_2)\bar\alpha_t-\lambda_2\widetilde\alpha_t\Bigr)^2
  &+\frac{1}{2}\Bigl((1-\lambda_2)\bar\alpha_t+\lambda_2\widetilde\alpha_t\Bigr)^2.
  \end{aligned}$}
  \end{equation}

Since $\widetilde X_t=\mathbb E[X_t]$ and
$\widetilde\alpha_t=\mathbb E[\alpha_t]$, we have
$\mathbb E[X_t\widetilde X_t]=\widetilde X_t^2$ and
$\mathbb E[\alpha_t\widetilde\alpha_t]=\widetilde\alpha_t^2$.
Hence, the individual--group cross terms are absorbed into
$\widetilde Q_t$ and $\widetilde R_t$, respectively, without invoking
the equilibrium consistency condition. The resulting two-parameter
model is therefore a specialization of the general PDE framework of
Sections~\ref{sec:extended-mfcg-model} and~\ref{sec:equilibrium}.

The corner values $(\lambda_1,\lambda_2)=(0,0)$ and $(1,1)$ recover,
respectively, the pure MFG regime, in which the global state and
control means are treated as exogenous benchmarks, and the pure MFC
regime, in which the corresponding group means are internalized.
The asymmetric corners $(1,0)$ and $(0,1)$ isolate internalization in
the state and control channels, respectively; in the numerical
experiment below, they produce stronger and weaker equilibrium
adjustment responses. Interior values combine the two effects, while
fixing one parameter and varying the other isolates the contribution
of the corresponding channel. Unlike the interpolation
in~\cite{carmona2023nash}, which is performed directly between cost
functionals, ours acts within the state and control benchmarks, as
in~\cite{dayanikli2025cooperation}. The code is available at \href{https://github.com/andrewjshi/2026-IEEECDC-LQMFCG}{https://github.com/andrewjshi/2026-IEEECDC-LQMFCG}.

\subsection{Experiment 1: A Verification Example}

Consider the dynamics \eqref{eq:lq-dynamics} with $A_t = 0, B_t = C_t = 1$:
\begin{equation} \label{eq:lq-dynamics-ex1}
	dX_t=\alpha_tdt + X_tdW_t.
\end{equation}

We use a running cost that retains the interpolation between MFG and MFC and provides a simple verification test.
Define the interpolated state benchmark
$
	\mu_t^\lambda=(1-\lambda)\bar X_t+\lambda\widetilde X_t
	\label{eq:interpolated-mean-ex1}
$
and let
\begin{equation}
	J(\alpha;\bar X,\bar\alpha)
	=
	\mathbb E\!\int_0^T
	\left[
		\alpha_t^2
		+\frac12\bigl(X_t-\mu_t^\lambda\bigr)^2
		+\frac12\bigl(\mu_t^\lambda\bigr)^2
	\right]dt.
	\label{eq:objective-ex1}
\end{equation}

The dynamics remain those in~\eqref{eq:lq-dynamics-ex1}.
Expanding~\eqref{eq:objective-ex1} in the general LQ cost gives
$
	Q_t=\frac12,
	\bar Q_t=(1-\lambda)^2,
	\widetilde Q_t=\lambda^2-\lambda,
$, 
$
	M_t=-(1-\lambda),
	\widetilde M_t=2\lambda(1-\lambda),
$
The remaining coefficients are chosen as follows.  We set $R_t=1$, while
$\widetilde R_t=\bar R_t=N_t=\widetilde N_t=S_t=\widetilde S_t=\bar S_t
=r_t=\widetilde r_t=\bar r_t=0$; all remaining state coefficients are zero.
We also take $H=\widetilde H=\bar H=0$.
Consequently, $\Lambda_0(t)=P_t$, $\Lambda_1(t)=\Pi_t$, and
$\Sigma_0(t)=\Sigma_1(t)=1$.
The Riccati equations \eqref{eq:P} and \eqref{eq:Pi} become
\begin{equation}
	\begin{cases}
		\dot P_t+P_t+\frac12-P_t^2=0, & P_T=0,\\
		\dot\Pi_t+P_t+\frac12-\lambda(1-\lambda)-\Pi_t^2=0,
		& \Pi_T=0,
	\end{cases}
	\label{eq:riccati-ex1}
\end{equation}
and  \eqref{eq:coupled_z_a} becomes
\begin{equation}
	\begin{cases}
		\dot z_t=-\Pi_tz_t-\phi_t, & z_0=X_0,\\
		\dot\phi_t=\Pi_t\phi_t
		-\dfrac12(1-\lambda)(2\lambda-1)z_t,
		& \phi_T=0,\\
		a_t=-\Pi_tz_t-\phi_t.
	\end{cases}
	\label{eq:coupled-ex1}
\end{equation}

We set $T = 1$ and $X_0 = 1$ and solve the system \eqref{eq:coupled-ex1} for five representative values $\lambda \in \{0, 0.25, 0.5, 0.75, 1\}$.
The mean state $z_t$ (Fig.~\ref{fig:experiment1}(a)) starts at $X_0 = 1$ for all $\lambda$.
Since $H = \tilde{H} = 0$, there is no terminal cost penalizing the final state, so agents have no incentive to drive $z_T$ toward any particular value.
The trajectories are therefore governed entirely by the running cost and decline throughout the horizon.
At $\lambda=0$ (pure MFG), the group planner treats the population mean $\bar X_t$ as a fixed external benchmark, and the mean state undergoes the shallowest decline.
As $\lambda$ increases toward $1$ (pure MFC), the group planner increasingly internalizes the benchmark penalty and applies a stronger control to reduce the mean state.
Consequently, the trajectories separate monotonically, with the $\lambda=1$ curve attaining the lowest mean state.

The mean control $a_t$ (Fig.~\ref{fig:experiment1}(b)) exhibits the same qualitative ordering.
All controls begin negative and rise toward zero at $T$, consistently with the absence of a terminal cost.
The MFG regime applies the least negative initial control, while the control becomes progressively more negative as $\lambda$ approaches the MFC regime.
This ordering reflects the group planner's increasing internalization of the cost associated with the state benchmark.

Finally, the Price of Anarchy (PoA) $J_{\mathrm{MFCG}}/J_{\mathrm{MFC}}$ (Fig.~\ref{fig:experiment1}(c)) decreases monotonically from approximately $1.081$ at $\lambda=0$ to $1$ at $\lambda=1$, and the curve is convex in $\lambda$.
The ratio equals $1$ at $\lambda=1$ by construction, since $\lambda=1$ corresponds to the cooperative optimum.
Thus the pure MFG regime incurs an aggregate efficiency loss of approximately eight percent relative to the MFC optimum.

\subsection{Experiment 2: A ``Two-Pillar'' Corporate Benchmarking Model}

We consider a stylized corporate benchmarking model with two interaction channels.
The state $X_t\in\mathbb R$ represents a firm's performance gap relative to a fixed target, while the control $\alpha_t\in\mathbb R$
represents its adjustment rate. Starting from a positive gap, a negative
control represents corrective action that reduces the gap. Firms are influenced
both by the performance gaps of their peers and by prevailing adjustment practices. 
This model has three distinct levels of interaction:
\textbf{Individual:} A representative firm's performance gap $X_t$
          and adjustment rate $\alpha_t$.
\textbf{Group (firm mean):} The within-firm means
          $\widetilde X_t=\mathbb E[X_t]$ and
          $\widetilde\alpha_t=\mathbb E[\alpha_t]$.
\textbf{Global (economy mean):} The economy-wide means
          $\bar X_t$ and $\bar\alpha_t$.

The representative firm's performance gap evolves as
    $$
    dX_t=\alpha_t\,dt+X_t\,dW_t,\qquad X_0=x_0.
    $$
The control changes the gap directly, while the multiplicative noise represents
uncertainty proportional to the current magnitude of the imbalance.

The model has two independently parameterized benchmarking pillars. The
\emph{state benchmark}
    $\mu_t^{\lambda_1}=(1-\lambda_1)\bar X_t
    +\lambda_1\widetilde X_t$
interpolates between an external, economy-wide performance benchmark and an
internal, group-level benchmark. Similarly, the \emph{adjustment benchmark}
    $\nu_t^{\lambda_2}=(1-\lambda_2)\bar\alpha_t
    +\lambda_2\widetilde\alpha_t$
interpolates between economy-wide adjustment practices and the group's internal
adjustment norm. The corner $(\lambda_1,\lambda_2)=(0,0)$ is the pure MFG regime,
whereas $(\lambda_1,\lambda_2)=(1,1)$ is the pure MFC regime. The two asymmetric
corners isolate the effects of internalizing only one benchmarking channel.

\begin{figure*}[t]
	\centering
	\begin{subfigure}[t]{0.24\textwidth}
		\centering
		\includegraphics[width=\textwidth]{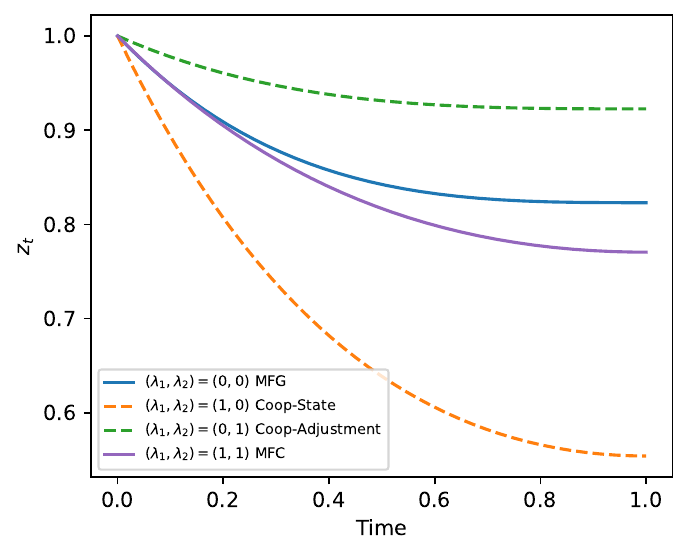}
	\end{subfigure}
	\hfill
	\begin{subfigure}[t]{0.24\textwidth}
		\centering
		\includegraphics[width=\textwidth]{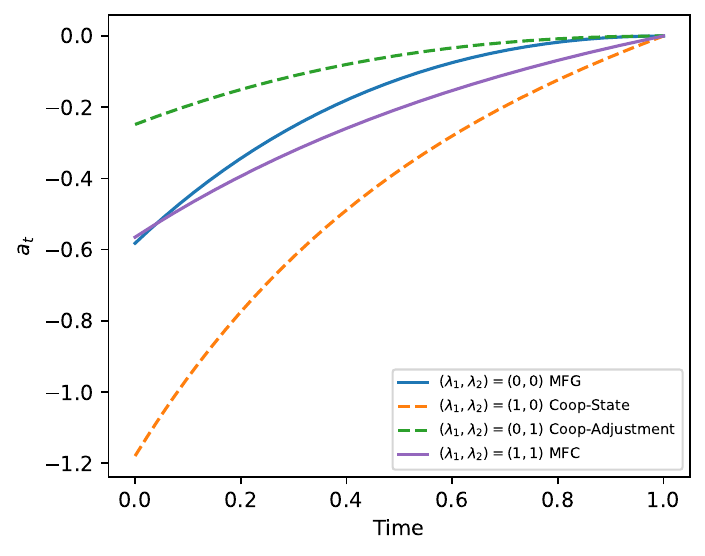}
	\end{subfigure}
	\hfill
	\begin{subfigure}[t]{0.24\textwidth}
		\centering
		\includegraphics[width=\textwidth]{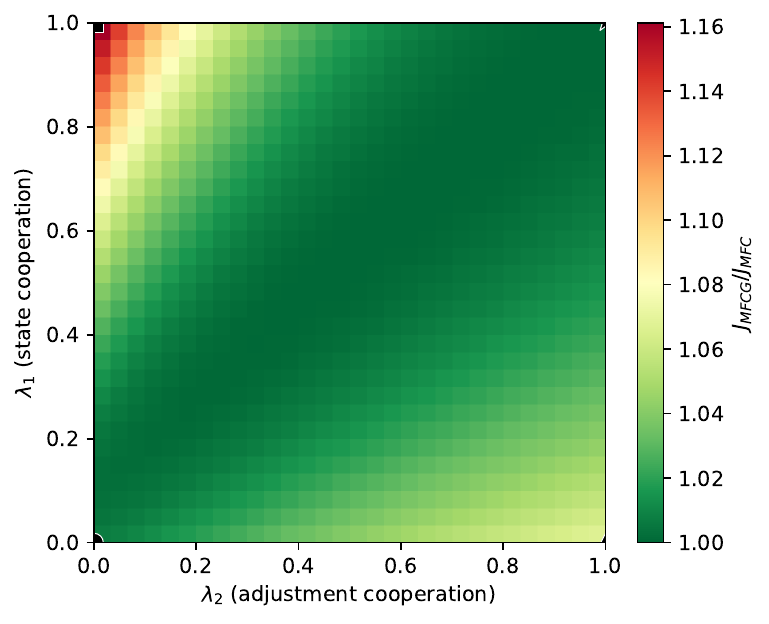}
	\end{subfigure}
	\caption{Experiment 2 for the four corner regimes and over
		the $(\lambda_1,\lambda_2)$ parameter square. Parameters: $c_\alpha=2$,
		$K_\alpha=3$, $K_x=2$, $T=1$, and $X_0=1$. Left: Mean performance gap $z_t$. Middle: Mean adjustment rate $a_t$. Right: PoA $J_{\mathrm{MFCG}}/J_{\mathrm{MFC}}$.}
	\label{fig:experiment2}
\end{figure*}

For the two-parameter experiment, 
we retain the benchmarks $\mu_t^{\lambda_1}$ and $\nu_t^{\lambda_2}$ %
and augment both deviation penalties by their corresponding benchmark penalties:
\begin{equation}\label{eq:running_cost2}
	\begin{split}
	f={}&\frac{c_\alpha}{2}\alpha^2
	+\frac{K_\alpha}{2}\left[(\alpha-\nu_t^{\lambda_2})^2
	+(\nu_t^{\lambda_2})^2\right]\\
	&+\frac{K_x}{2}\left[(x-\mu_t^{\lambda_1})^2
	+(\mu_t^{\lambda_1})^2\right].
	\end{split}
\end{equation}
The term $c_\alpha\alpha^2/2$ is the direct implementation cost of adjustment.
The deviation terms penalize departures from prevailing adjustment and state
benchmarks, while the squared-benchmark terms penalize an aggressive collective
adjustment norm and a large collective performance gap. Here $c_\alpha>0$ is the
quadratic adjustment-cost coefficient, while $K_\alpha>0$ and $K_x>0$ weight the
adjustment- and state-benchmark penalties.
Moreover, along any symmetric equilibrium, every parameter pair
is evaluated using the same welfare criterion,
\begin{equation}\label{eq:welfare-ex2}
	J=\int_0^T\left[\frac{c_\alpha+K_\alpha}{2}\,\mathbb E[\alpha_t^2]
	+\frac{K_x}{2}\,\mathbb E[X_t^2]\right]dt.
\end{equation}

For \eqref{eq:running_cost2}, the nonzero state coefficients are
$Q_t=\tfrac{K_x}{2},\quad
\widetilde Q_t=K_x(\lambda_1^2-\lambda_1),\bar Q_t=K_x(1-\lambda_1)^2,
M_t=-K_x(1-\lambda_1),
\widetilde M_t=2K_x\lambda_1(1-\lambda_1),$ 
and the nonzero control coefficients are
$R_t=\tfrac{c_\alpha+K_\alpha}{2},\quad
\widetilde R_t=K_\alpha(\lambda_2^2-\lambda_2),
\bar R_t=K_\alpha(1-\lambda_2)^2,
N_t=-K_\alpha(1-\lambda_2),\quad
\widetilde N_t=2K_\alpha\lambda_2(1-\lambda_2).$
Let $\Delta_{\lambda_2}=c_\alpha+K_\alpha(1-2\lambda_2+2\lambda_2^2)$.
The Riccati equations become
$\dot P_t+P_t+\frac{K_x}{2}-\frac{2P_t^2}{c_\alpha+K_\alpha}=0,
P_T=0,$ and 
$\dot\Pi_t+P_t+K_x\left(\frac12-\lambda_1(1-\lambda_1)\right)
	-\frac{2\Pi_t^2}{\Delta_{\lambda_2}}=0,
\Pi_T=0.
$
The consistency system is
\begin{equation}\label{eq:coupled-ex2}
	\begin{cases}
	\dot z_t=a_t,\quad z_0=x_0,\\[3pt]
	\begin{aligned}
	\dot\phi_t={}&\frac{2\Pi_t}{\Delta_{\lambda_2}}\phi_t
	+\frac{K_\alpha(1-\lambda_2)(2\lambda_2-1)\Pi_t}
	{\Delta_{\lambda_2}}a_t\\
	&-\frac{K_x}{2}(1-\lambda_1)(2\lambda_1-1)z_t,
	\qquad \phi_T=0,
	\end{aligned}\\[5pt]
	\displaystyle
	a_t=-\frac{2}{c_\alpha+K_\alpha\lambda_2}(\Pi_tz_t+\phi_t).
	\end{cases}
\end{equation}
Letting $V_t=\operatorname{Var}(X_t)$, the centered feedback and variance satisfy
	$\alpha_t-a_t=-\frac{2P_t}{c_\alpha+K_\alpha}(X_t-z_t),$
$\dot V_t=\left(1-\frac{4P_t}{c_\alpha+K_\alpha}\right)V_t+z_t^2,
	V_0=0.$
The criterion \eqref{eq:welfare-ex2} can then be evaluated as
\begin{equation}\label{eq:cost-ex2}
	\begin{split}
	J=\int_0^T\biggl[&\frac{c_\alpha+K_\alpha}{2}a_t^2
	+\frac{c_\alpha+K_\alpha}{2}
	\left(\frac{2P_t}{c_\alpha+K_\alpha}\right)^2V_t\\
	&+\frac{K_x}{2}(z_t^2+V_t)\biggr]dt.
	\end{split}
\end{equation}

We take $T=X_0=1$, $c_\alpha=2$, $K_\alpha=3$, and $K_x=2$.
Figure~\ref{fig:experiment2} compares the four corner regimes. The
Coop-State, Competitive-Adjustment regime $(1,0)$ produces the fastest
decline in the mean gap and the most negative adjustment, whereas
$(0,1)$ produces the weakest response; the MFG and MFC trajectories lie
between these asymmetric cases. The PoA is approximately
$1.007$ at the MFG corner, reaches approximately $1.16$ near $(1,0)$,
and equals $1$ at the MFC corner. Thus, partial cooperation need not
improve welfare monotonically: internalizing the state objective without
the adjustment channel can induce excessive correction.

\section{Concluding Remarks}

In this paper, we formulate a framework for extended MFCG problems that provides a unified model of cooperative groups interacting in a non-cooperative environment through the joint distribution of states and actions.
To characterize the equilibrium, we derive the associated HJB-KFP system. Focusing on an LQ model, we provide optimality conditions through ODEs. 
We complement this study with numerical experiments. 
We note that our current LQ setting is not fully general; therefore, future work will study more general LQ models from both theoretical and numerical perspectives.
We will also study other models beyond the LQ setting. 
In particular, we intend to numerically solve the PDE system \eqref{eq:MFCG-PDE}. 
Furthermore, the incorporation of common noise and hierarchical network structures remains a significant area for exploration.
These directions are reserved for future research.

{\footnotesize
\bibliographystyle{abbrv}
\bibliography{references}
}

\appendix

\section{Auxiliary result for the PDE system}

\begin{lemma}[Necessary condition of optimality for the individual problem]
	\label{lem:PDE-extended-MFC}
	Let $\bar{v}$ be the control used by the population.
	Given $\bar{\varphi}^{\bar{v}}$, $\bar{\gamma}^{\bar{v}}$, and $\bar{\psi}^{\bar{v}}_T$, if $v^*$ is an optimal feedback control minimizing \eqref{eq:objective-functional-1} subject to \eqref{eq:KFP-eqn}, then, for all $(t,x)$, it satisfies the following necessary optimality condition:
	\begin{align}
		 & \partial_v f(t, x, v^*(t,x), \tilde{\gamma}^{v^*,\bar{v}}_t, \bar{\gamma}_t^{\bar{v}})
		\label{eq:optimality-condition-MFC}
		\\
		 & + D_xu^{v^*,\bar{v}}(t,x) \cdot \partial_v \drift(t,x,v^*(t,x), \tilde{\varphi}^{v^*,\bar{v}}_t, \bar{\varphi}_t^{\bar{v}})
		\notag
		\\
		 & + \partial_v \gamma(t,x, v^*(t,x))\cdot K^\gamma_t + \partial_v \varphi(t,x, v^*(t,x))\cdot K^\varphi_t
		= 0,
		\notag
	\end{align}
	where $(m^{v^*,\bar{v}}, u^{v^*,\bar{v}})$ solves the following PDE system:
	\begin{equation}
		\label{eq:MFC-PDE}
		\begin{cases}
			\partial_t m^{v^*,\bar{v}}(t,x) + A^* m^{v^*,\bar{v}}(t,x)
			\\
			\qquad + \operatorname{div}(\drift(t,x,v^*(t,x),\tilde{\varphi}^{v^*,\bar{v}}_t, \bar{\varphi}_t^{\bar{v}}) m^{v^*,\bar{v}}(t,x)) = 0,
			\\ m^{v^*,\bar{v}}(0, x) = m_0(x),
			\\
			-\partial_t u^{v^*,\bar{v}}(t,x) + Au^{v^*,\bar{v}}(t,x)
			\\
			\qquad - \drift(t,x,v^*(t,x),\tilde{\varphi}^{v^*,\bar{v}}_t, \bar{\varphi}_t^{\bar{v}}) \cdot D_xu^{v^*,\bar{v}}(t,x)
			\\
			\qquad - f(t,x,v^*(t,x), \tilde{\gamma}^{v^*,\bar{v}}_t, \bar{\gamma}_t^{\bar{v}})
			\\
			\qquad - \varphi(t,x, v^*(t,x)) \cdot K^\varphi_t
			- \gamma(t,x, v^*(t,x)) \cdot K^\gamma_t
			= 0,
			\\
			u^{v^*,\bar{v}}(T, x) = \tcost(x, \tilde{\psi}_T^{v^*,\bar{v}}, \bar{\psi}_T^{\bar{v}}) + \psi(x) \cdot K^\psi_T,
		\end{cases}
	\end{equation}
	with $K^\gamma_t, K^\varphi_t$, and $K^\psi_T$ given by:
	\begin{align*}
		K^\gamma_t
		 & = \int_{\R^d} \partial_{\tilde{\gamma}} f(t, x', v^*(t,x'), \tilde{\gamma}^{v^*,\bar{v}}_t, \bar{\gamma}_t^{\bar{v}}) m^{v^*,\bar{v}}(t,x') dx', \\
		K^\varphi_t
		 & = \int_{\R^d} D_xu^{v^*,\bar{v}}(t, x') \cdot
		\\
		 & \qquad\qquad \partial_{\tilde{\varphi}} \drift(t,x',v^*(t,x'), \tilde{\varphi}^{v^*,\bar{v}}_t, \bar{\varphi}_t^{\bar{v}}) m^{v^*,\bar{v}}(t, x') dx',
		\\
		K^\psi_T
		 & = \int_{\R^d} \partial_{\tilde{\psi}} \tcost(x', \tilde{\psi}_T^{v^*,\bar{v}}, \bar{\psi}_T^{\bar{v}}) m^{v^*,\bar{v}}(T, x') dx'.
	\end{align*}
\end{lemma}

Before providing the proof, a few remarks are in order.

\begin{remark}[Terminology]
	Note that, although the equation for $u$ is often referred to as the HJB equation in the MFC literature, it is \emph{not} the Bellman equation of the MFC problem, and the function $u$ is \emph{not} the value function of the player; see, e.g.,~\cite{bensoussan2017interpretation} for more details on the connection between the Bellman equation and the HJB equation.
\end{remark}

\begin{remark}[Consistency with standard MFC]
	If the individual player does not interact with the population (i.e., the functions $\drift$, $f$, and $\tcost$ are constant with respect to $\bar{\varphi}, \bar{\gamma}$, and $\bar{\psi}$, respectively), and if, furthermore, the action distribution is not involved (i.e., the functions $\varphi$, $\gamma$, and $\psi$ are constant with respect to the action), then the problem reduces to a standard MFC problem. 
    Consistent with this, the PDE system in
	Lemma~\ref{lem:PDE-extended-MFC} reduces to the standard MFC PDE system given, e.g., in \cite[eq.~(4.12)]{bensoussan2013mean}.
\end{remark}

\section{Proofs}

We present here proofs that were omitted above.

\begin{proof}[Proof of Lemma~\ref{lem:PDE-extended-MFC}]
	Let $v^*$ be an optimal control. 
    Consider the perturbation $v^*+ \theta w$ and compute the G\^ateaux derivative
	\begin{multline} \label{eq:J-variation}
		\frac{d}{d\theta}J(v^* + \theta w; \bar{v})\bigg|_{\theta=0} \\
		= \int_0^T \int_{\R^d} (\partial_{\tilde{\gamma}}f \cdot \check{\gamma}_t\cdot m^{v^*, \bar{v}} + \partial_vf \cdot w  \cdot m^{v^*, \bar{v}} + f\cdot \check{m})dxdt \\
		+ \int_{\R^d} (\partial_{\tilde{\psi}}\tcost\cdot \check{\psi}_T\cdot m^{v^*,\bar{v}}(T,x) + \tcost\cdot\check{m}(T,x))dx,
	\end{multline}
	where $\check{m}$, $\check{\gamma}_t$, $\check{\varphi}_t$, and $\check{\psi}_T$ are the first variations given by
	\begin{equation} \label{eq:first-variation}
		\begin{aligned}
			 & \partial_t \check{m} + A^*\check{m} + \operatorname{div}(\drift\check{m})
			\\
			 & \quad + \operatorname{div}\left( (\partial_{\tilde{\varphi}}\drift\cdot \check{\varphi}_t + \partial_v \drift\cdot w)m^{v^*,\bar{v}} \right) = 0, \check{m}(0,x) = 0, \\
			 & \check{\gamma}_t = \int_{\R^d} \gamma(t,\xi, v^*)\check{m}(t,\xi)d\xi
			\\
			 & \qquad+ \int_{\R^d} \partial_v \gamma(t,\xi, v^*) w(t,\xi)m^{v^*,\bar{v}}(t,\xi)d\xi,
			\\
			 & \check{\varphi}_t = \int_{\R^d} \varphi(t,\xi, v^*)\check{m}(t,\xi)d\xi
			\\
			 & \qquad + \int_{\R^d} \partial_v \varphi(t,\xi, v^*) w(t,\xi)m^{v^*,\bar{v}}(t,\xi)d\xi,
			\\
			 & \check{\psi}_T = \int_{\R^d} \psi(\xi) \check{m}(T,\xi)d\xi.
		\end{aligned}
	\end{equation}
	Substituting the variations $\check{\gamma}_t$ and $\check{\psi}_T$ into the G\^ateaux derivative and using the definitions of $K_t^{\gamma}$ and $K_T^{\psi}$, we have
	\begin{align}
		 & \quad \int_0^T \int_{\R^d} \partial_{\tilde{\gamma}} f \cdot \check{\gamma}_t \cdot m^{v^*,\bar{v}} dx dt                                                      \\
		 & = \int_0^T \int_{\R^d} K^\gamma_t \cdot \left( \gamma(t,x, v^*) \check{m} + \partial_v \gamma(t,x, v^*) w m^{v^*,\bar{v}} \right) dx dt                              \\
		 & \text{and}\quad \int_{\R^d} \partial_{\tilde{\psi}} \tcost \cdot \check{\psi}_T \cdot m^{v^*, \bar{v}}(T, x) dx = \int_{\R^d} K^\psi_T \cdot \psi(x) \check{m}(T, x) dx.
	\end{align}
	It follows that
	\begin{equation} \label{eq:J-derivative}
		\begin{aligned}
			 & \quad\frac{d}{d\theta}J(v^* + \theta w; \bar{v})\bigg|_{\theta=0}                                                                                                                           \\
			 & = \int_0^T \int_{\R^d} \left[ \left( \partial_v f + K^\gamma_t \partial_v \gamma \right) w m^{v^*,\bar{v}} + \left( f + K^\gamma_t \cdot \gamma \right) \check{m} \right] dx dt \nonumber \\
			 & \quad + \int_{\R^d} \left( \tcost + K^\psi_T \cdot \psi \right) \check{m}(T,x) dx.
		\end{aligned}
	\end{equation}
	To eliminate $\check{m}$, we take the inner product of the equation of $u$ in~\eqref{eq:MFC-PDE} with $\check{m}$ over $\R^d \times [0, T]$. 
    Applying integration by parts and substituting the equation of $\check{m}$ in \eqref{eq:first-variation} yields
	\begin{multline}
		\int_0^T \int_{\R^d} \left( f + \gamma \cdot K^\gamma_t \right) \check{m} dx dt = -\int_{\R^d} u(T,x) \check{m}(T,x) dx \\
		+ \int_0^T \int_{\R^d} \left( D_xu \cdot \partial_v \drift + K^\varphi_t \partial_v \varphi \right) w m^{v^*,\bar{v}} dx dt.
	\end{multline}
	Using this relation, the cost variation simplifies to
	\begin{align*}
		 & \frac{d}{d\theta}J(v^* + \theta w; \bar{v})\bigg|_{\theta=0}
		\\
		 & = \int_0^T \int_{\R^d} \big[ \partial_v f + D_xu \cdot \partial_v \drift
		\\
		 & \qquad + \partial_v \gamma K^\gamma_t + \partial_v \varphi K^\varphi_t \big] w m^{v^*,\bar{v}} dx dt
		\\
			 & \qquad + \int_{\R^d} \left( \tcost + K^\psi_T \cdot \psi - u(T,x) \right) \check{m}(T,x) dx.
	\end{align*}
	The terminal condition of the equation for $u$ in~\eqref{eq:MFC-PDE} dictates that the second integral vanishes. 
    Since $v^*$ is optimal, $\frac{dJ}{d\theta}\big|_{\theta=0} = 0$ for an arbitrary variation $w$.
    By the fundamental lemma of the calculus of variations, the bracketed term must be identically zero almost everywhere, which yields the optimality condition~\eqref{eq:optimality-condition-MFC}.
\end{proof}

\begin{proof}[Proof of Theorem~\ref{thm:general_mfcg}]
	Let $\hat{v}$ be an equilibrium control. 
    By the optimality condition in Definition~\ref{def:equilibrium}, $\hat{v}$ is an optimal feedback control that minimizes \eqref{eq:objective-functional-1} subject to \eqref{eq:KFP-eqn}. 
    Applying Lemma~\ref{lem:PDE-extended-MFC} with $v^*=\hat{v}$, we find that $\hat{v}$ satisfies \eqref{eq:optimality-condition-MFC}.
    Furthermore, since both the population and the individual player use the control $\hat{v}$, we have $\tilde{\varphi}_t^{v^*,\bar{v}} = \tilde{\varphi}_t^{\hat{v},\hat{v}} = \bar{\varphi}_t^{\hat{v}}$, $\tilde{\gamma}_t^{v^*,\bar{v}} = \tilde{\gamma}_t^{\hat{v},\hat{v}} = \bar{\gamma}_t^{\hat{v}}$, and $\tilde{\psi}_T^{v^*,\bar{v}} = \tilde{\psi}_T^{\hat{v},\hat{v}} = \bar{\psi}_T^{\hat{v}}$.
    Hence, equations~\eqref{eq:optimality-condition-MFC}--\eqref{eq:MFC-PDE} rewrite as equations~\eqref{eq:optimality-condition-MFCG}--\eqref{eq:MFCG-PDE}. 
    This completes the proof.
\end{proof}

\begin{proof}[Proof of Theorem~\ref{thm:lq_mfcg}]
	The proof is divided into two steps.

	\textbf{Step 1: Solving the individual MFC with a fixed population evolution.}
	In the linear-quadratic framework, fixing the distribution flow means that we fix $(\bar{X}_t)_{t\in[0,T]}$ and $(\bar{v}_t)_{t\in[0,T]}$. 
    Viewing them as given processes, we obtain an MFC problem with slight modifications to the coefficients of the linear terms:
	\begin{align}
		\begin{aligned}\label{eq:qqrr}
			 & \check{q}_t(\bar{X})=\frac 1 2 M_t\bar{X}_t+q_t, \quad \check{\tilde{q}}_t(\bar{X})=\frac 1 2 \tilde{M}_t\bar{X}_t+\tilde{q}_t, \\
			 & \check{r}_t(\bar{v})=\frac 1 2 N_t\bar{v}_t+r_t, \quad \check{\tilde{r}}_t(\bar{v})=\frac 1 2 \tilde{N}_t\bar{v}_t+\tilde{r}_t.
		\end{aligned}
	\end{align}
	The six additional terms in the cost functionals $F$ and $G$ (see~\eqref{eq:LQ-costs}), namely,
	$ \bar Q_t\bar{X}_t\cdot\bar{X}_t ,\bar{R}_t \bar{v}_t\cdot\bar{v}_t, 2 \bar{S}_t \bar{X}_t\cdot \bar{v}_t, 2\bar{q}_t\cdot\bar{X}_t , 2 \bar{r}_t\cdot \bar{v}_t ,  \bar{H} \bar{X}_T\cdot \bar{X}_T$,
	are not affected by the individual player and hence do not affect the optimal control. 
    By \cite[Theorem 2.6]{graber2016linear}, the optimal control is given by\footnote{Equations (2.31) and (2.32) of \cite[Theorem 2.6]{graber2016linear} are mutually inconsistent: the optimal control (2.32) contains the term $(B^\top+\bar{B}^\top)\phi$, whereas the trajectory equation (2.31), obtained by substituting (2.32) into the state dynamics, does not. Consistency requires the $\phi$-terms to be retained and requires $\phi$ to solve a backward ODE with $\phi_T=0$ rather than the forward integral stated after (2.27). We use the corrected system throughout. An omitted term in the expression for $\Sigma_0$ and an inaccuracy in the equation for $\phi$ are also reported in \cite[Footnote 7]{delsarto2024pollution}.}
	\begin{align}\label{eq:opti_control}
		 & \hat{v}_t(\bar{X},\bar{v})
		=-\Sigma_0^{-1}(t)\Lambda_0(t)(X_t-\tilde{X}_t)
		\\
		 & \qquad -\Sigma_1^{-1}(t) \big(\Lambda_1(t)\tilde{X}_t+\check{r}_t(\bar{v})+\check{\tilde{r}}_t(\bar{v})+B^\top_t\phi_t(\bar{X},\bar{v})\big),
	\end{align}
	where $\phi_t(\bar{X},\bar{v})$ satisfies
	{\small \begin{align}\label{eq:phi}
\dot\phi_t(\bar{X},\bar{v})=&-\bigl(A_t^\top-\Lambda_1^\top\Sigma_1^{-1}B_t^\top\bigr)\phi_t-\frac{1}{2}(M^\top_t+\tilde{M}^\top_t)\bar{X}_t \\
&+\Lambda_1^\top(t)\Sigma_1^{-1}(t)\big(r_t+\tilde r_t+\frac{1}{2}(N^\top_t+\tilde{N}^\top_t)\bar{v}_t\big)-(q_t+\tilde q_t),\\
\phi_T(\bar{X},\bar{v})= &0.
	\end{align}}
    Here, $P$ and $\Pi$ are the unique solutions of equations~\eqref{eq:P} and~\eqref{eq:Pi}, respectively.  
    Their existence and uniqueness are guaranteed by \cite[Theorem 2.6]{graber2016linear} under Assumption~\ref{assum:coeff}.

	\textbf{Step 2: Matching the consistency condition at equilibrium.}
	From Step 1, we know that under a given population mean flow $(\bar{X},\bar{v})$, the mean flow of the optimal controlled intra-group state $\tilde{X}$ satisfies the ODE
	\begin{equation}
		\begin{cases}
			\dot{\tilde{X}}_t=\big(A_t-B_t\Sigma_1^{-1}(t)\Lambda_1(t)\big)\tilde{X}_t \\
			\quad - B_t \Sigma_1^{-1}(t)\big(\check{r}_t(\bar{v}_t)+\check{\tilde{r}}_t(\bar{v}_t)
			+ B_t^\top\phi_t\big),
			\\
			\tilde{X}_0= \int_{\RR^d}x m_0(x)dx,
		\end{cases}
		\label{eq:opti_dynamics}
	\end{equation}
	and the mean process of the optimal control is
	\begin{align}
		\tilde{v}_t=\mathbb{E}[\hat{v}_t(\bar{X},\bar{v})]
		 & =-\Sigma_1^{-1}(t)\big(\Lambda_1(t)\tilde{X}_t+\frac 1 2(N_t+\tilde{N}_t)\bar{v}_t \notag
		\\
		 & \quad \quad +r_t+\tilde{r}_t+B^\top_t\phi_t(\bar{X},\bar{v})\big).
		\label{eq:opti_ealpha}
	\end{align}
	In view of \eqref{eq:phi}, \eqref{eq:opti_dynamics} and \eqref{eq:opti_ealpha}, it is seen that \eqref{eq:coupled_z_a} is obtained by letting for any $t\in[0,T]$, $\tilde{X}_t=\bar{X}_t, \tilde{v}_t=\bar{v}_t$. Then if we take $(z_t)_{t\in[0,T]}$ and $(a_t)_{t\in[0,T]}$ as the given mean processes of the state and of the control for the whole population respectively, we must have $\tilde{X}_t=z_t, \tilde{v}_t=a_t$.
    Hence the consistency condition is satisfied.
\end{proof}

\end{document}